\documentclass[10pt]{amsart} 
\usepackage{amsmath,amssymb, mathtools}
\usepackage{upref}
\usepackage{color}
\usepackage[colorlinks=true, linkcolor=blue, citecolor=red] {hyperref}
\usepackage[makeroom]{cancel}
\usepackage[margin=3cm]{geometry}
\usepackage{dsfont}
\usepackage{dirtytalk}
\makeatletter
\@namedef{subjclassname@2020}{%
  \textup{2020} Mathematics Subject Classification}
\makeatother

\usepackage{stackengine,graphicx}

\usepackage{esint}

\newtheorem{definition}{Definition}[section]
\newtheorem{theorem}{Theorem}[section]
\newtheorem{lemma}{Lemma}[section]

\newtheorem{proposition}{Proposition}[section]

\newtheorem{remark}{Remark}[section]
\newtheorem*{maintheorem*}{Main Theorem}
\allowdisplaybreaks
\numberwithin{equation}{section}

\renewcommand{\i}{\ifmmode\mathit{\mathchar"7010 }\else\char"10 \fi}
\renewcommand{\j}{\ifmmode\mathit{\mathchar"7011 }\else\char"11 \fi}
\newcommand{\R}{\mathbb{R}}

\renewcommand{\d}{\mathrm{d}}

\DeclareMathOperator*{\supp}{supp}

\renewcommand{\div}{\operatorname{div}}
\newcommand{\curl}{\operatorname{curl}}

\renewcommand{\d}{\mathrm{d}}
\newcommand{\dd}{\mathrm{\,d}}

\newcommand{\DD}{\Delta}

\newcommand{\1}{\mathds{1}}

\DeclareMathOperator*{\dvg}{\mbox{div}}

\def\Xint#1{\mathchoice
{\XXint\displaystyle\textstyle{#1}}%
{\XXint\textstyle\scriptstyle{#1}}%
{\XXint\scriptstyle\scriptscriptstyle{#1}}%
{\XXint\scriptscriptstyle\scriptscriptstyle{#1}}%
\!\int}
\def\XXint#1#2#3{{\setbox0=\hbox{$#1{#2#3}{\int}$ }
\vcenter{\hbox{$#2#3$ }}\kern-.6\wd0}}

\def\avint{\Xint-}

\usepackage{todonotes}

\title[Liouville theorems for 2-D MHD]{On Liouville-type theorems for the 2D stationary MHD equations}

\subjclass[2020]{35Q35, 76W05, 35B65, 35B53.}
\keywords{Liouville theorem; incompressible magneto-hydrodynamics (MHD)}

\author[N. De Nitti]{Nicola De Nitti}
\address[N. De Nitti]{Friedrich-Alexander-Universität Erlangen-Nürnberg, Department of Mathematics, Chair of Applied Analysis (Alexander von Humboldt Professorship), Cauerstr. 11, 91058 Erlangen, Germany.}
\email{nicola.de.nitti@fau.de}

\author[F. Hounkpe]{Francis Hounkpe}
\address[F. Hounkpe]{Mathematical Institute, University of Oxford, Radcliffe Observatory Quarter, Woodstock Road, Oxford, OX2 6GG, UK.}
\email{francishounkpe@gmail.com}

\author[S. Schulz]{Simon Schulz}
\address[S. Schulz]{Faculty of Mathematics, University of Cambridge, Wilberforce Road, Cambridge, CB3 0WA, UK.}
\email{sms79@cam.ac.uk}

\begin{document}

\begin{abstract}
We establish new Liouville-type theorems for the two-dimensional stationary magneto-hydrodynamic incompressible system assuming that the velocity and magnetic field have bounded Dirichlet integral. The key tool in our proof is observing that the stream function associated to the magnetic field satisfies a simple drift-diffusion equation for which a maximum principle is available.
\end{abstract}

\maketitle

\section{Introduction}\label{sec:introduction}
\subsection{Liouville-type results in fluid mechanics}\label{ssec:intro_liouville}
We are interested in studying Liouville-type properties for the 2D incompressible stationary \emph{Magneto-Hydrodynamic (MHD)} system on the whole plane $\R^2$:
\begin{equation}\label{eq:mhd}
    \left\{
    \begin{gathered}
    -\DD u + u\cdot\nabla u + \nabla \pi = b\cdot \nabla b,\quad \dvg u = 0\\
    -\DD b + u\cdot\nabla b - b\cdot\nabla u = 0
    \end{gathered}
    \right.\quad\mbox{in }\R^2
\end{equation}
In system \eqref{eq:mhd}, $u,b:\R^2 \to \R^2$ stand for the velocity and the magnetic field respectively, and $\pi:\R^2 \to \R$ is the pressure; this system and its time-dependent analogue are used to model electrically conductive fluids such as plasma, liquid metals, electrolytes etc. For more physical background and mathematical theory, we refer to \cite{Schn09} and the references contained therein. Here, unlike what is done in the literature, we do not assume that the magnetic field $b$ is divergence free (which expresses the \say{endless} nature of magnetic streamlines); we were able to show instead, that this property can be obtained directly from the system \eqref{eq:mhd} (see end of this section for further discussion).

In this work, we are interested in Liouville-type results for solutions to \eqref{eq:mhd} with finite Dirichlet energy, i.e.,
\begin{equation}\label{setup}
    \int_{\R^2}\left(|\nabla u|^2 + |\nabla b|^2\right)\dd x <\infty.
\end{equation}

When $b\equiv 0$ in system \eqref{eq:mhd}, i.e., in the case of the incompressible stationary \emph{Navier-Stokes} system
\begin{equation}\label{eq:stationary_NS}
\left\{
\begin{gathered}
- \Delta u + u \cdot \nabla u + \nabla \pi = 0, \\
\div u = 0,
\end{gathered}
\right.
\end{equation}
the question of the triviality of $u$ under the above condition was first solved in \cite{Gilb78}; this triviality was also established later in \cite{Koch09} provided that the velocity $u$ is just bounded. Those two results rely heavily on the fact that one has a nice equation for the vorticity $w := \partial_1 u_{2} - \partial_2 u_{1}$, i.e.,
\begin{equation}\label{eq:vorticity ns}
- \DD w + u\cdot\nabla w = 0.
\end{equation}
In higher dimensions, we refer to the works \cite{Chae16_3,Galdi11,Ser16,Ser19}.

Note that, when the magnetic field $b$ is not null, the equation \eqref{eq:vorticity ns} for the vorticity is no longer available, which adds to the difficulty of the Liouville problem for the MHD system. Indeed, the aforementioned Liouville-type results for the incompressible stationary Navier-Stokes system still remain unsolved for the 2D MHD system.

In recent years, many mathematicians attempted to bring a complete understanding to the Liouville problem for the MHD system; we cite the works of \cite{Chae16,Chae18,Schulz19,Zhang15} (and the references contained therein) for interesting results. We also point out the recent contribution of Wang and Wang in \cite{Wang19}, where they proved a Liouville theorem for the system in question under some smallness condition on the $L_1$-norm of $b$.

Our work is more in the direction of this latter result, \cite{Wang19}. The novelty of the present paper is that we uncover a nice drift-diffusion equation for the stream function associated to the magnetic field, and for which a maximum principle is available. This allows us to bring new insights to this problem and to improve existing results.

As mentioned above, it is very common to see in the literature an added incompressibility condition \say{$\div b = 0$} on the magnetic field; however, by doing so, the system becomes, from a mathematical point of view, over-determined. One of the key observations in this work is that under condition \eqref{setup}, the incompressibility of the magnetic field $b$ can be derived directly from the system; this is done in Proposition \ref{prop:stream} below. 



\subsection{Preliminaries and notation}
\label{ssec:prelim}
In this section, we recall the definition of various function spaces that will be central to our analysis and explain the notations we use throughout this work.

To begin with, we say that a function $f$ belongs to the space $BMO(\R^2)$ if the following quantity is finite:
\[
\|f\|_{BMO(\R^2)} := \sup\left\{ \frac{1}{|B(0,r)|}\int_{B(x_0,r)} |f-[f]_{x_0,r}| \dd x : B(x_0,r) \subset \R^2 \right\},
\]
where we denote the average of $f$ over the ball $B(x_0,r)$ centred at $x_0$ and with radius $r$ as
\[
[f]_{x_0,R} := \avint_{B(x_0,R)}f(x) \dd x \left(=\frac{1}{|B(R)|}\int_{B(x_0,R)} f(x) \dd x\right),\quad [f]_{,R} := [f]_{0,R};
\]
more generally, the average of $f$ over a bounded subset $\Omega$ is denoted by $[f]_{\Omega} := \frac{1}{|\Omega|}\int_\Omega f \, \d x$.
See for instance \cite{grafakos,Stein93} for useful properties of the BMO space. In what follows, we denote the unit ball by $B$ and the ball with radius $R$ centred at the origin by $B(R)$.

We say that a vector field $v$ belongs to the space $BMO^{-1}(\R^2)$, if there exists a matrix $F\in BMO(\R^2;\R^{2 \times 2})$ such that $v = \dvg F$, and this is the only property of this function space we need in this work; see for instance \cite{Koch01} for discussions regarding this space.

Next, we say that a function $f$ is in the Hardy space $\mathcal{H}^1(\R^2)$ if the following quantity is finite:
\[
\|f\|_{\mathcal{H}^1(\R^2)} := \|\mathcal{M}f\|_{L_1(\R^2)},
\]
where $\mathcal{M}f(x) := \sup_{\phi \in \mathcal{F}}\sup_{t>0} |(\phi_t\star f)(x)|$, with
\[
\mathcal{F} := \left\{ \phi \in C^{\infty}_0(B): |\nabla \phi| \leq 1 \right\}
\]
and $\phi_t(x) := t^{-n}\phi(x/t)$. See also \cite{grafakos,Stein93} and \cite{cmsHardy} for important results concerning this function space.


We use $c$ or $C$ to denote an absolute constant and write $C(A,B,\ldots)$ when the constant depends on the parameters $A,B,\ldots$. This constant may change from line to line.

Finally, the notation `$\perp$' will be used to refer to the orthogonal of a vector field, i.e., for any $v=(v_1,v_2) \in \mathbb{R}^2$, its orthogonal is written $v^\perp = (v_2,-v_1)$. Likewise, for any scalar function $\phi$, we use the notation $\nabla^\perp \phi = (\partial_2 \phi, -\partial_1 \phi)$ to refer to the orthogonal of its gradient. Furthermore, since we constrain ourselves to the two-dimensional case, the quantity $\curl v$ for any vector field is a scalar function given by $\curl v := \nabla \cdot v^\perp  = \partial_1 v_{2} - \partial_2 v_{1}$.

\subsection{Outline of the paper}
\label{ssec:summary}


The paper is organised as follows. In Section \ref{sec:main}, we state our main Liouville-type theorems precisely. The key idea for proving our results is the observation that there exists a stream function $\psi$ for $b$ such that $b = \nabla^\perp \psi$ and the last equation in \eqref{eq:mhd} can be rewritten as
	\[
	\nabla^{\perp}(-\DD \psi + u\cdot \nabla \psi) = 0.
	\]
In Section \ref{sec:stream}, we prove the incompressibility and existence of this stream function associated to the magnetic field and further properties. In Section \ref{sec:proof_liouville}, we use these tools to prove our main results. Finally, in Appendix \ref{sec:appendix}, we collect some technical lemmas used at various points in the paper.

\section{Main results}\label{sec:main}
We use the following definition of weak solution for \eqref{eq:mhd}.
\begin{definition}[Weak solutions]\label{def:solutions}
	We say that the pair $(u, b) \in \mathcal{D}'(\R^2)\times \mathcal{D}'(\R^2)$ (here $\mathcal{D}'(\R^2)$ denotes the space of distributions on $\R^2$) is a weak solution to the 2D MHD equation \eqref{eq:mhd} in a domain $\Omega \subset \mathbb{R}^{2}$ provided that:
	
	\begin{enumerate}
		\item $\nabla u, \nabla b \in L_{2,loc}(\Omega)$;
		\item $\operatorname{div} u=0$ in the sense of distributions;
		\item the couple \((u, b)\) satisfies
	\[
	\begin{gathered}
	\int_{\Omega} \nabla u \cdot \nabla \varphi \dd  x-\int_{\Omega}u\otimes u : \nabla \varphi \dd x=\int_{\Omega}(b \cdot \nabla b) \cdot \varphi \dd x, \\
	\int_{\Omega} \nabla b \cdot \nabla \varphi \dd x -\int_{\Omega} b\otimes u : \nabla \varphi \dd x=\int_{\Omega}(b \cdot \nabla u) \cdot \varphi \dd x
	\end{gathered}
	\]
	for all \(\varphi \in C_{c}^{\infty}(\Omega)\) with \(\varphi=\left(\varphi_{1}, \varphi_{2}\right)\) and \(\operatorname{div} \varphi=0\).
	\end{enumerate}
\end{definition}

Note that the requirement that $\nabla u , \nabla b \in L_{2,loc}(\mathbb{R}^2)$ implies $u,b\in L_{2,loc}(\mathbb{R}^2)$ (\textit{cf.}~\cite[Chapter 1, Theorem 1.1]{Ser14}), so the previous definition is reasonable.

\begin{remark}[Regularity of weak solutions]
Standard regularity results  (see for instance \cite[Chapter 3.2]{Ser14}) yield that any pair $(u,b) \in \mathcal{D}'(\R^2)\times \mathcal{D}'(\R^2)$ of weak solutions to \eqref{eq:mhd}, say in $\R^2$, is actually smooth, i.e.,
\[
u,b \in C^{\infty}(\R^2).
\]
\end{remark}

Our first main result is stated as follows.
\begin{theorem}\label{thm:main}
Let $u,b$ be weak solutions to system \eqref{eq:mhd} in $\R^2$ such that condition \eqref{setup} holds. Assume in addition that one of the following conditions holds:
\begin{enumerate}
    \item $b\in L_2(\R^2)$;
    \item $u\in L_p(\R^2)$ and $b\in L_q(\R^2)$ with $p,q \in ]2,\infty[$ such that
    \begin{equation}\label{B.4}
        \frac{2}{\max(p,q)} + \frac{1}{p} \geq \frac{1}{2};
    \end{equation}
    \item  $u\in L_p(\R^2)$ and $b\in L_q(\R^2)$ with $1/p + 1/q = 1$ and $p \in [1,2[$.
\end{enumerate}
Then, $u$ and $b$ are constants.
\end{theorem}

\begin{remark}
Under condition \eqref{setup}, we have $b\in BMO(\R^2)$ 
and, thanks to the interpolation between $L_p$ and $BMO$ spaces (see e.g.~\cite[Theorem 2]{Chen-Zhu05}), we can replace the first point of Theorem \ref{thm:main}
by $b\in L_p(\R^2)$, with $p\in [1,2]$.
\end{remark}



Our second result addresses the case where we allow the velocity $u$ to grow or rapidly oscillate.
\begin{theorem}\label{thm:main2}
Let $u,b$ be weak solutions to system \eqref{eq:mhd} in $\R^2$ such that condition \eqref{setup} holds. Assume in addition that the following condition holds:
\[
u\in BMO^{-1}(\R^2).
\]
Then,
\begin{enumerate}
    \item if $b\in L_q(\R^2)$ with $q\in ]2,\infty]$, we have $u,b\equiv 0$;
    \item or, alternatively, if $b$ belongs also to $BMO^{-1}(\R^2)$, we have that $u,b\equiv 0$.
\end{enumerate}
\end{theorem}

\begin{remark}
The second point in the previous theorem was also proved in \cite{Chae18} in the 3D case; the proof of this result presented herein is simpler.
\end{remark}


\section{Properties of the stream function associated to the magnetic field}
\label{sec:stream}
In this section, we establish the incompressibility of the magnetic field $b$ and,  as a result, the existence of a stream function associated to it.

\begin{proposition}[Incompressibility and existence of the stream function associated to the magnetic field] \label{prop:stream}
Let $u,b$ weak solutions to system \eqref{eq:mhd} such that condition \eqref{setup} holds. Then, $\div b \equiv 0$ and there exists a stream function $\psi$
such that $b = \nabla^\perp \psi$.
\end{proposition}

In order to prove it, we need the following simple Liouville-type result.
\begin{lemma}[Liouville-type theorem for drift-diffusion equations]\label{lm:liouville}
Let $\rho\in L_{\infty}(\R^2)$ and let $u \in \mathring{L}^1_2(\R^2)$ (the closure of $C^{\infty}_0(\R^2)$ with respect to the semi-norm $\|\nabla \cdot \|_{L_2(\R^2)}$) be a vector field such that $\div(\rho u) = 0$. Let $f\in L_2(\R^2)$ be a solution of
\begin{equation}\label{B.3}
    -\DD f + \dvg(\rho u f) = 0 \quad\mbox{in }\mathcal{D}'(\R^2).
\end{equation}
Then $f \equiv 0$.
\end{lemma}

\begin{proof}

Let us start by noticing that classical regularity theory insures that $f\in C^1(\R^2)$. Next, we introduce the following cut-off function: $\varphi \in C^{\infty}_0(\R^2)$ such that $0\leq \varphi \leq 1$ with $\varphi \equiv 1$ in $B(1/2)$ and $\varphi \equiv 0$ in $B\setminus B(3/4)$; let $R>1$ and set $\varphi_R(x) := \varphi(x/R)$. By testing equation \eqref{B.3} with $f\varphi_R^2$ and then integrating by parts twice, we get
\begin{align*}
    \int_{B(R)}|\nabla f|^2\varphi_R^2 \dd x &= \int_{B(R)}\frac{|f|^2}{2}\DD \varphi_R^2 \dd x + \int_{B(R)}|f|^2\varphi_R\rho u\cdot \nabla \varphi_R \dd x\\
    &\leq \frac{c(\varphi)}{R^2}\int_{B(R)\setminus B(R/2)}|f|^2 \dd x + \frac{c(\varphi)|[u]_{,R}|\|\rho\|_{L_{\infty}}}{R}\int_{B(R)\setminus B(R/2)}|f|^2 \dd x\\
    & + \frac{c(\varphi)\|\rho\|_{L_{\infty}}}{\sqrt{R}}\left(\int_{B(R)} |f|^2 \dd x \right)^{\frac{1}{2}}\left(\int_{\R^2}|f\varphi_R|^4 \dd x\right)^{\frac{1}{4}}\left(\avint_{B(R)}|u - [u]_{,R}|^4 \dd x \right)^{\frac{1}{4}}
\end{align*}
Now we use Ladyzhenskaya's inequality to obtain:
\begin{align*}
    \left(\int_{\R^2}|f\varphi_R|^4 \dd x\right)^{\frac{1}{4}} &\leq c\left(\int_{\R^2} |f\varphi_R|^2\dd x \right)^{\frac{1}{4}}\left(\int_{\R^2} |\varphi_R\nabla f + f\nabla\varphi_R|^2\dd x \right)^{\frac{1}{4}}\\
    &\leq c\left(\int_{B(R)}|\nabla f|^2\varphi_R^2 \dd x \right)^{\frac{1}{4}}\left( \int_{B(R)}|f|^2 \dd x \right)^{\frac{1}{4}} + \frac{c(\varphi)}{\sqrt{R}}\left(\int_{B(R)}|f|^2 \dd x\right)^{\frac{1}{2}}.
\end{align*}
Next, by using Lemma \ref{lm:lions} to control the term $|[u]_{,R}|$ and by introducing for simplicity the norm $|\|u\|| := \|\nabla u\|_{L_2(\R^2)} + \|u\1_B\|_{L_2(\R^2)}$, we find
\begin{align*}
    \int_{B(R)}|\nabla f|^2\varphi_R^2 \dd x &\leq c(\varphi)\left(R^{-2} + \|\rho\|_{L_{\infty}(\R^2)}R^{-1}\sqrt{\log 2 R}|\|u\|| \right)\int_{B(R)\setminus B(R/2)}|f|^2 \dd x\\
    & + \frac{c(\varphi)\|\rho\|_{L_{\infty}}}{\sqrt{R}}\|\varphi_R\nabla f\|_{L_2(B(R))}^{\frac{1}{2}}\|f\|_{L_2(B(R))}^{\frac{3}{2}}\|\nabla u\|_{L_2(B(R))}\\
    & + \frac{c(\varphi)\|\rho\|_{L_{\infty}}}{R}\|f\|^2_{L_2(B(R))}\|\nabla u\|_{L_2(B(R))}\\
    &\leq \frac{1}{2} \int_{B(R)}|\nabla f|^2\varphi_R^2 \dd x + c(\varphi)\left(R^{-2} + \|\rho\|_{L_{\infty}}R^{-1}\sqrt{\log 2 R}|\|u\|| \right)\|f\|^2_{L_2}\\
    & + c(\varphi)\left(R^{-\frac{2}{3}}\|\rho\|_{L_{\infty}}^{\frac{4}{3}}\|\nabla u\|^{\frac{4}{3}}_{L_2(\R^2)} + R^{-1}\|\rho\|_{L_{\infty}}\|\nabla u\|_{L_2(\R^2)} \right)\|f\|^2_{L_2(\R^2)};
\end{align*}
and it is clear from the latter inequality that $\int_{\R^2}|\nabla f|^2 \dd x  = 0$ and, a fortiori, $f \equiv 0$.
\end{proof}

\begin{proof}[Proof of Proposition \ref{prop:stream}]
Taking the divergence in the second equation in \eqref{eq:mhd}, we get:
\[
-\DD(\dvg b) + u\cdot \nabla (\dvg b) = 0\quad\mbox{in }\R^2.
\]
Applying Lemma \ref{lm:liouville} with $f = \dvg b$, $\rho = 1$, we get that $\dvg b \equiv 0$, and the existence of the stream function $\psi$ comes from the solvability of the equation $-\DD \psi = \curl b$, which is clear since the right-hand side of this latter equation belongs to $L_2(\mathbb{R}^2)$ due to \eqref{setup}.
\end{proof}

\begin{remark}[Liouville theorem for the incompressible Navier-Stokes system]\label{R3.1}
Another direct consequence of Lemma \ref{lm:liouville} is that, when $b\equiv 0$ (i.e.~for the incompressible Navier-Stokes system) and $\nabla u \in L_2(\R^2)$, by noticing that the vorticity $w := \curl u$ satisfies \eqref{eq:vorticity ns}, we get that $w \equiv 0$; thus $-\DD u = 0$ and, consequently, $u$ is constant. We have thus recovered, in a simpler way, the result of Gilbarg and Weinberger in \cite{Gilb78}.

\end{remark}


Now, let $\psi$ be as in Proposition \ref{prop:stream}. Then, through explicit computations, the second equation in \eqref{eq:mhd} may be rewritten precisely as:
\[
\nabla^\perp\left( -\DD \psi + u\cdot \nabla \psi \right) = 0.
\]
Consequently, we get that, for some (a priori unknown) constant $c_0 \in \mathbb{R}$,
\begin{equation}\label{B.5}
    -\DD \psi + u\cdot \nabla \psi = c_0.
\end{equation}
The next proposition provides more information on the value of $c_0$.

\begin{proposition}\label{prop:c0}
Let $u,b$ be solutions to system \eqref{eq:mhd} such that condition \eqref{setup} holds. If we assume in addition that $u$ or $b$ belongs to $BMO^{-1}(\R^2)$, then $c_0 = 0$ (with $c_0$ as in \eqref{B.5}).
\end{proposition}

\begin{remark}\label{rem:assumptionC0}
It will be clear from the proof that follows that the conclusion of this proposition still holds if we assume for instance $u\in L_p(\R^2)$ and $b\in L_q(\R^2)$ with $p,q\in [1,\infty[$; of course, this condition can be weakened even further.
\end{remark}

\begin{proof}
We will present only the case $b\in BMO^{-1}(\R^2)$ since the case $u\in BMO^{-1}(\R^2)$ can be dealt with in a similar manner.

Recall that the Riesz transform of a function $f$ is defined as $Rf := \nabla ((-\DD)^{\frac{1}{2}}f$. Because the Riesz transform is a bounded operator from $BMO(\R^n)$ to $BMO(\R^n)$ (see for instance \cite{Stein93}), we have that the stream function $\psi$ associated to $b$ belongs to $BMO(\R^2)$.

Recall our cut-off function $\varphi \in C^{\infty}_0(\R^2)$ such that $0\leq \varphi \leq 1$ with $\varphi \equiv 1$ in $B(1/2)$ and $\varphi \equiv 0$ in $B\setminus B(3/4)$. As before, for $R>1$, set $\varphi_R(x) := \varphi(x/R)$. We get from \eqref{B.5} that
\[
c_0\int_{B(R)} \varphi_R \dd x = -\int_{B(R)}\DD \psi\varphi_R \dd x + \int_{B(R)} u\cdot \nabla (\psi - [\psi]_{,R})\varphi_R \dd x.
\]
This implies,
\begin{align*}
    |c_0|R^2\int_B\varphi \dd x &\leq R\|\nabla b\|_{L_2}\|\varphi\|_{L_2} + \left| \int_{B(R)}(\psi - [\psi]_{,R}) u\cdot\nabla \varphi_R \dd x  \right|\\
    &\leq c(\varphi)\left(R\|\nabla b\|_{L_2(\R^2)} + \|\psi\|_{BMO(\R^2)} \|u\|_{L_2(B(R))}\right).
\end{align*}
From Lemma \ref{lm:lions}, we get that
\[
|c_0| \leq c(\varphi,b,u)\left( \frac{1}{R} + \frac{\sqrt{\log{2R}}}{R}\right) \to 0\mbox{ as }R\to \infty.
\]
This concludes the proof.
\end{proof}

\section{Proof of the main theorems}\label{sec:proof_liouville}
We are ready to give the proof of Theorem \ref{thm:main}.

\begin{proof}[Proof of Theorem \ref{thm:main}]
We divide the proof into four steps: the first two steps deal, respectively, with the first two points in the theorem, and the third point in the theorem is proved in the last two steps.

\vspace{2.5mm}
\paragraph{\textbf{Step 1.}} Since $b\in L_2(\R^2)$ implies $b \in BMO^{-1}(\R^2)$, we have thanks to Proposition \ref{prop:c0} that
\[
-\DD \psi + u\cdot\nabla \psi = 0\quad\mbox{in }\R^2.
\]
By setting $\psi^R(x) := \psi(R x)$ ($R>0$), we see that $\psi^R$ satisfies the conditions in Lemma \ref{lm:oscillation}, where we note that the maximum principle property is inherited from the previous drift-diffusion equation for $\psi$. Consequently, by taking $r=1/2$ in the oscillation lemma, we obtain:
\[
\sup_{x\in B(1/2)}|\psi^R(x) - \psi(0)| \leq c\|\nabla\psi^R\|_{L_2(B\setminus B(1/2))},
\]
thus
\[
\sup_{x\in B(R/2)}|\psi(x) - \psi(0)| \leq c\|b\|_{L_2(B(R)\setminus B(R/2))}\to 0\mbox{ as }R\to \infty.
\]
This implies that $\psi$ is constant and therefore $b\equiv 0$; consequently we have that
\[
-\DD u + u\cdot\nabla u + \nabla \pi = 0,\quad \dvg u = 0\quad\mbox{in }\R^2,
\]
with $\nabla u\in L_2(\R^2)$. This implies, thanks to Remark \ref{R3.1}, that $u$ is constant.

\vspace{2.5mm}
\paragraph{\textbf{Step 2.}} By applying the divergence operator to the first equation in \eqref{eq:mhd}, we find that
\begin{equation}\label{eq:laplace pressure}
-\DD \pi = \dvg\dvg(u\otimes u - b\otimes b)\quad\mbox{in }\R^2;
\end{equation}
this guarantees, if we set $r = \max(p,q)$, that
\begin{equation}\label{B.6}
    \|\pi\|_{L_{\frac{r}{2}}(\R^2)} \leq c\left( \|u\|^2_{L_r(\R^2)} + \|b\|^2_{L_r(\R^2)}\right);
\end{equation}
to obtain estimate \eqref{B.6}, we used the hypothesis, the embedding of $u,b$ in $BMO(\R^2)$ (due to \eqref{setup}) and the interpolation between $L_s$ and $BMO$ spaces; see \cite[Theorem 2]{Chen-Zhu05}. A detailed justification of this estimate (in a more general context) is given later in the document (see `Step 1 (Pressure estimates)' of the proof of the first point in Theorem \ref{thm:main2}).

Next, we rewrite the first equation in \eqref{eq:mhd}, in the following way:
\[
-\DD u + \nabla\left(\frac{|u|^2}{2} + \pi -\frac{|b|^2}{2}\right) -u^{\perp}\curl u = -b^{\perp}\curl b.
\]
By taking the scalar product with $u$ in the previous line, we get that
\[
-\DD\frac{|u|^2}{2} + |\nabla u|^2 + u\cdot\nabla\left(\frac{|u|^2}{2} + \pi -\frac{|b|^2}{2}\right) = -u\cdot b^{\perp}\curl b.
\]
Meanwhile, in view of the equality $\nabla \psi = -b^{\perp}$ and thanks to \eqref{B.5} (together with Remark \ref{rem:assumptionC0}), we have that
\[\curl b = u\cdot b^{\perp}\quad \mbox{in }\R^2.\]
Consequently, returning to the previous elliptic equation, we have that
\begin{equation}\label{B.7}
    |\nabla u|^2 + |\curl b|^2 = \DD\frac{|u|^2}{2} - u\cdot\nabla\left(\frac{|u|^2}{2} + \pi -\frac{|b|^2}{2}\right)\quad \mbox{in }\R^2.
\end{equation}
For ease of notation, set $Q := |u|^2/2 + \pi -|b|^2/2 \in L_{\frac{r}{2}}(\R^2)$. By multiplying \eqref{B.7} by our usual rescaled cut-off function $\varphi_R$ and then integrating, we obtain that
\begin{align*}
    \int_{B(R/2)}(|\nabla u|^2 + |\curl b|^2) \dd x &\leq \int_{B(R)\setminus B(R/2)}\frac{|u|^2}{2}\DD \varphi_R \dd x + \int_{B(R)\setminus B(R/2)} u\cdot\nabla \varphi_R Q \dd x\\
    &\leq  \frac{c(\varphi,p)}{R^{\frac{4}{p}}}\left(\int_{B(R)\setminus B(R/2)}|u|^p \dd x\right)^{\frac{2}{p}} + \\
    & + \frac{c(\varphi,p,q)}{R^{1-\frac{2}{s}}}\left(\int_{B(R)\setminus B(R/2)}|u|^p \dd x\right)^{\frac{1}{p}}\left(\int_{B(R)\setminus B(R/2)}|Q|^{\frac{r}{2}} \dd x\right)^{\frac{2}{r}},
\end{align*}
with $1/s = 1 - 1/p - 2/r$. Thus, the right-hand side in last inequality above vanishes in the limit as $R\to\infty$ whenever $1 - 2/s\geq 0$, i.e., $1/p + 2/r \geq 1/2$. And the second point is proved.

\vspace{2.5mm}
\paragraph{\textbf{Step 3.}} In this penultimate step, we prove the third point in the theorem for the non-critical cases $p \in ]1,2[$. Firstly, observe that, for this range of exponents, the div-curl lemma of Coifman, Lions, Meyer and Semmes (\textit{cf.}~\cite[Theorem 4]{cmsHardy}) applies. Due to the incompressibility condition, we deduce that $u \cdot \nabla \psi$ belongs to the Hardy space $\mathcal{H}^1(\mathbb{R}^2)$. Now, define
\begin{equation}\label{eq:v def approx psi in Step 3}
    v := (-\Delta)^{-1} (-u \cdot \nabla \psi),
\end{equation}
i.e.,
\begin{equation*}
    v(x) = \frac{1}{2\pi}\int_{\mathbb{R}^2} \log |x-y| (u\cdot \nabla \psi)(y) \, \d y, \qquad \forall x \in \mathbb{R}^2.
\end{equation*}
Observe that the above integral is well-defined since the function $y \mapsto \log |y|$ belongs to the space $BMO(\mathbb{R}^2)$, which is the dual of $\mathcal{H}^1(\mathbb{R}^2)$; in fact, we readily deduce that $v \in L_\infty(\mathbb{R}^2)$. Moreover,
an application of Lebesgue's dominated convergence theorem shows that
\begin{equation*}
    \nabla v(x) = \frac{1}{2\pi} \int_{\mathbb{R}^2} \frac{(x-y)}{|x-y|^2} (u \cdot \nabla \psi)(y) \, \d y, \qquad \forall x \in \mathbb{R}^2.
\end{equation*}
The above and, for instance, \cite[Theorem 2.2]{Krantz} and \cite[Theorem 2.1.2]{grafakos}, yield $\nabla v \in L_2(\mathbb{R}^2)$. Additionally, using the fundamental property of the Poisson kernel in \eqref{eq:v def approx psi in Step 3}, we obtain, in the sense of distributions,
\begin{equation*}
    -\Delta v + u \cdot \nabla \psi = 0 \qquad \text{in } \mathbb{R}^2.
\end{equation*}

Our goal is now to ``transfer'' the square-integrability of $\nabla v$ onto $\nabla \psi$, since these satisfy similar equations. To this end, returning to the drift-diffusion equation for the stream function $\psi$ and taking an additional derivative, we find that
\begin{equation*}
    -\Delta (\nabla \psi - \nabla v) = 0 \qquad \text{in } \mathbb{R}^2.
\end{equation*}
Since $\nabla \psi - \nabla v \in L_{p'}(\R^2) + L_2(\R^2)$,
we deduce that $\nabla \psi - \nabla v = 0$ in $\R^2$. Hence, we get that $b \in L_2(\mathbb{R}^2)$ and conclude as we did in `Step 1' of this proof.

\vspace{2.5mm}
\paragraph{\textbf{Step 4.}} Finally, we prove the third point in the theorem for the critical case $p=1$. In this case, we cannot apply the div-curl lemma of Coifman, Lions, Meyer and Semmes. Nevertheless, we can still construct a solution of the problem $-\Delta v + u \cdot \nabla \psi = 0$ with the desired properties, i.e., $v \in L_\infty(\mathbb{R}^2)$ with $\nabla v \in L_2(\mathbb{R}^2)$. This is encapsulated in Lemma \ref{lemma:construction}, the proof of which is delayed until the Appendix. Comparing the equation for $v$ with the drift-diffusion equation for the stream function $\psi$, we find once again
\begin{equation*}
    -\Delta(\psi - v) = 0 \qquad \text{in } \mathbb{R}^2.
\end{equation*}
By applying the operator \say{$\nabla^{\perp}$} to the previous equation and by noticing that $b - \nabla^{\perp} v \in L_{\infty}(\R^2) + L_2(\R^2)$, we obtain that
there exists a constant $a \in \mathbb{R}^2$ such that
\begin{equation*}
    b = a + \nabla^{\perp} v \qquad \mbox{in } \mathbb{R}^2.
\end{equation*}
From the previous identity, we also get that $\nabla v \in L_2(\R^2)\cap L_{\infty}(\R^2)$, which leads,
by interpolation, to
\begin{equation}\label{eq:grad v very integrable}
    \nabla v \in L_q(\mathbb{R}^2), \qquad \forall q \in [2,\infty].
\end{equation}
We find, in this case, the following equation for the pressure:
\[
-\DD \pi = \dvg\dvg(u\otimes u - a\otimes \nabla^\perp v - \nabla^\perp v\otimes a - \nabla^\perp v\otimes \nabla^\perp v),
\]
which leads, without loss of generality, to
\begin{equation}
    \|\pi\|_{L_r(\R^2)} \leq c\|u\otimes u - a\otimes \nabla^\perp v - \nabla^\perp v\otimes a - \nabla^\perp v\otimes \nabla^\perp v\|_{L_r(\R^2)},
\end{equation}
for all $r\in [2,\infty[$ (more details can be found in `Step 1 (Pressure estimates)' of the proof of the first point in Theorem \ref{thm:main2}).

We can see that identity \eqref{B.7} still holds true in our current setting; notice now that
\[
\nabla \frac{|b|^2}{2} = \nabla \frac{|\nabla v|^2}{2} + \nabla (a\cdot \nabla^\perp v).
\]
Thus, by setting $\tilde{Q} := |u|^2/2 + \pi - |\nabla v|^2/2 - a\cdot \nabla^\perp v \in L_r(\R^2)$ with $r \in [2,\infty[$, we have that
\[
|\nabla u|^2 + |\curl b|^2 = \DD\frac{|u|^2}{2} - u\cdot\nabla \tilde{Q} \quad \mbox{in }\R^2.
\]
Consequently, by testing the previous equation against our usual rescaled cut-off function $\varphi_R$, we obtain
\begin{align*}
    \int_{B(R/2)}(|\nabla u|^2 + |\curl b|^2) \dd x &\leq \int_{B(R)\setminus B(R/2)}\frac{|u|^2}{2}\DD \varphi_R \dd x + \int_{B(R)\setminus B(R/2)} u\cdot\nabla \varphi_R \tilde{Q} \dd x\\
    &\leq  \frac{c(\varphi)}{R^{2}} \int_{B(R)\setminus B(R/2)}|u|^2 \dd x + \\
    & + \frac{c(\varphi)}{R}\left(\int_{B(R)\setminus B(R/2)}|u|^2 \dd x\right)^{\frac{1}{2}}\left(\int_{B(R)\setminus B(R/2)}|\tilde{Q}|^2 \dd x\right)^{\frac{1}{2}},\\
    & \to 0\quad \mbox{ as }R\to \infty.
\end{align*}
This proves the final point of the theorem.
\end{proof}

Let us give now the proof of Theorem \ref{thm:main2}. We start with the second assertion in the theorem.

\begin{proof}[Proof of the second point in Theorem \ref{thm:main2}]
Because $u,b\in BMO^{-1}(\R^2)$, we have
that $u = \nabla^{\perp}\phi$ and $b = \nabla^{\perp} \psi$ with $\phi,\psi\in BMO(\R^2)$. Thanks, to Proposition \ref{prop:stream}, we know that
\[
-\DD \psi + u\cdot \nabla \psi = 0;
\]
by multiplying the above equation by $(\psi - [\psi]_{,R})\varphi_R^2$ (where $\varphi_R$ is our usual rescaled cut-off function), and then integrating by parts, we get that
\begin{equation*}
    \int_{B(R)}|\nabla \psi|^2 \varphi_R^2 \dd x = \frac{1}{2}\int_{B(R)}|\psi - [\psi]_{,R}|^2\DD \varphi^2_R \dd x + 2\int_{B(R)}\varphi_R\nabla\psi\cdot\nabla^{\perp}\varphi_R (\psi - [\psi]_{,R})(\phi - [\phi]_{,R})\dd x;
\end{equation*}
this implies, thanks to Young's inequality,
\begin{align*}
    \int_{B(R/2)}|\nabla \psi|^2 \dd x &\leq c(\varphi)\avint_{B(R)}|\psi - [\psi]_{,R}|^2 \dd x + \frac{c(\varphi)}{R^2}\int_{B(R)}|\psi - [\psi]_{,R}|^2|\phi - [\phi]_{,R}|^2 \dd x\\
    &\leq c(\varphi)(\|\psi\|^2_{BMO(\R^2)} + \|\psi\|^2_{BMO(\R^2)}\|\phi\|^2_{BMO(\R^2)}).
\end{align*}
Thus, by taking the limit $R\to \infty$, we get that $\nabla \psi \in L_2(\R^2)$, i.e., $b\in L_2(\R^2)$. By applying the first point of Theorem \ref{thm:main}, we conclude the proof.
\end{proof}

We prove now the first assertion in Theorem \ref{thm:main2}
\begin{proof}[Proof of the first point in Theorem \ref{thm:main2}]
We divide the proof into two steps. Let us also recall that, since $u \in BMO^{-1}(\R^2)$, we have that $u = \nabla^{\perp} \phi$ with $\phi \in BMO(\R^2)$.

\vspace{2.5mm}
\paragraph{\textbf{Step 1 (Pressure estimates).}} Let us start by noticing that, under our hypothesis, we have that:
\[
-\DD u + u\cdot\nabla u \in (\mathring{L}^1_2(\R^2))^*
\]
the dual of $\mathring{L}^1_2(\R^2)$ (the closure of $C^{\infty}_0(\R^2)$ with respect of the semi-norm $\|\nabla \cdot \|_{L_2(\R^2)}$). To see this, consider $\varphi \in C^{\infty}_0(\R^2)$; then, we have that
\begin{align*}
    \langle -\DD u_i + u\cdot \nabla u_i,\varphi \rangle &= \int_{\R^2}\nabla u_i \cdot \nabla \varphi \dd x - \int_{\R^2} \varphi \nabla u_i\cdot \nabla^{\perp} \varphi \dd x\\
    &\leq c\left(\|\nabla u\|_{L_2(\R^2)} + \|\varphi\|_{BMO(\R^2)}\|\nabla u\|_{L_2(\R^2)}\right)\|\nabla \varphi\|_{L_2(\R^2)}\quad(\forall i = 1,2),
\end{align*}
where, for the last inequality, we used the duality between $BMO$ and Hardy $\mathcal{H}^1$ spaces, and the div-curl lemma of Coifman, Lions, Meyer and Semmes (see for instance \cite[Theorem 4]{cmsHardy}). Consequently, we have that (see for instance \cite[Theorem II.8.2]{Galdi11}) there exists $F\in L_2(\R^2;\R^{2\times 2})$ such that
\begin{equation}\label{B3.1}
    -\DD u + u\cdot\nabla u = \dvg(F);
\end{equation}
in particular, we deduce that
\begin{equation}\label{B3.2}
    \nabla \pi \in (\mathring{L}^1_2(\R^2))^* + L_s(\R^2),
\end{equation}
with $1/s = 1/q + 1/2$. Next, let us introduce $\pi_1$ and $\pi_2$ defined as follow:
\[
\pi_1 := (-\DD)^{-1}\dvg\dvg(F)\quad\mbox{and}\quad \pi_2 := -(-\DD)^{-1}\dvg\dvg(b\otimes b),
\]
where $F$ is as in \eqref{B3.1}.
We have that $\pi_1 \in L_2(\R^2)$ and $\pi_2\in L_{\frac{q}{2}}(\R^2)$ if $q\neq \infty$ or $\pi_2 \in BMO(\R^2)$ if $q = \infty$, and $\nabla \pi_2 \in L_s(\R^2)$ (with $1/s = 1/q + 1/2$). Going back to the first equation in \eqref{eq:mhd}, we get that
\[
-\DD(\pi - \pi_1 - \pi_2) = 0;
\]
but since $\nabla(\pi - \pi_1 - \pi_2) \in (\mathring{L}^1_2(\R^2))^* + L_s(\R^2)$ (see \eqref{B3.2} and definition of $\pi_1$ and $\pi_2$), we have that $\nabla(\pi - \pi_1 - \pi_2) \equiv 0$; therefore, without loss of generality
\begin{equation}
    \pi = \pi_1 + \pi_2.
\end{equation}

\vspace{2.5mm}
\paragraph{\textbf{Step 2.}} Similarly to what we did in `Step 2' of the proof of Theorem \ref{thm:main}, we obtain that (see \eqref{B.7}):
\begin{equation}\label{E4.7}
    |\nabla u|^2 + |\curl b|^2 = \DD\frac{|u|^2}{2} - u\cdot\nabla\left(\frac{|u|^2}{2} + \pi_1 + \pi_2 -\frac{|b|^2}{2}\right)\quad \mbox{in }\R^2.
\end{equation}
Before going any further, let us introduce the following special mean function:
\[
[u]_{\varphi,R} := \int_{B(R)}u(x)\varphi_R(x) \dd x\left( \int_{B(R)}\varphi_R(x)\dd x \right)^{-1},
\]
where $\varphi_R$ is our usual rescaled cut-off function. Now, we multiply $\varphi_R$ and integrate over the ball $B(R)$ to obtain
\begin{align*}
    &\int_{B(R/2)}\left( |\nabla u|^2 + |\curl b|^2\right)\dd x \leq \frac{1}{2}\int_{B(R)\setminus B(R/2)}|u|^2\DD \varphi_R \dd x\\
    & + \int_{B(R)\setminus B(R/2)}(\phi - [\phi]_{,R})\nabla (|u|^2/2)\cdot\nabla^{\perp} \varphi_R \dd x + \int_{B(R)\setminus B(R/2)} \pi_1 u\cdot\nabla \varphi_R \dd x\\
    & + \int_{B(R)\setminus B(R/2)}(\phi - [\phi]_{,R})\nabla \pi_2\cdot\nabla^{\perp} \varphi_R \dd x + \int_{B(R)\setminus B(R/2)}(\phi - [\phi]_{,R})\nabla (|b|^2/2)\cdot\nabla^{\perp} \varphi_R \dd x\\
    & =: \sum_{i = 1}^5 I_i(R).
\end{align*}
Our goal now is to estimate the $I_i$ terms. We have
\begin{align*}
    I_1(R) &= \frac{1}{2}\int_{B(R)\setminus B(R/2)}| u - [u]_{B(R)\setminus B(R/2)}|^2\DD \varphi_R \dd x\\
    &\phantom{=} + [u]_{B(R)\setminus B(R/2)}\cdot\int_{B(R)\setminus B(R/2)}(u - [u]_{B(R)\setminus B(R/2)})\DD \varphi_R \dd x, \\
    &\leq \frac{c(\varepsilon,\varphi)}{R^2}\int_{B(R)\setminus B(R/2)}| u - [u]_{B(R)\setminus B(R/2)}|^2 \dd x + 2\varepsilon(\left|[u]_{B(R)\setminus B(R/2)} - [u]_{\varphi,R}\right|^2 + \left|[u]_{\varphi,R}\right|^2)\\
    &\quad \leq c(\varepsilon,\varphi)\int_{B(R)\setminus B(R/2)}|\nabla u|^2 \dd x + \varepsilon c(\varphi) \int_{B(R)}|\nabla u|^2 \dd x + \varepsilon\frac{c(\varphi)}{R^2}\|\phi\|^2_{BMO(\R^2)}.
\end{align*}
Let us also point out that, for this last inequality, we used
\begin{equation}\label{eq:funny poincare}
    \begin{aligned}
    \left|[u]_{B(R)\setminus B(R/2)} - [u]_{\varphi,R}\right|^2 &\leq c(\varphi)\avint_{B(R)}|u - [u]_{B(R)\setminus B(R/2)}|^2 \dd x\\
    &\leq c(\varphi)\int_{B(R)}|\nabla u|^2 \dd x;
    \end{aligned}
\end{equation}
and we bound the term $|[u]_{\varphi,R}|$ as follows:
\begin{align*}
    |[u]_{\varphi,R}| &= \left|\frac{1}{R^2\int_B \varphi(x)\dd x}\int_{B(R)}\nabla^{\perp}(\phi - [\phi]_{,R})\varphi_R \dd x\right|\\
    &\leq \frac{c(\varphi)}{R}\avint_{B(R)}|\phi - [\phi]_{,R}|\dd x\\
    &\leq \frac{c(\varphi)}{R}\|\phi\|_{BMO(\R^2)}.
\end{align*}

Next, we have
\begin{align*}
    I_2(R) &= \int_{B(R)\setminus B(R/2)}(\phi - [\phi]_{,R})\nabla (|u - [u]_{B(R)\setminus B(R/2)}|^2/2)\cdot\nabla \varphi_R \dd x\\
    & \phantom{=} + [u]_{B(R)\setminus B(R/2)}\cdot\int_{B(R)\setminus B(R/2)}(\phi - [\phi]_{,R})\nabla u\nabla^{\perp}\varphi_R \dd x\\
    & \leq c(\varphi)\left(\int_{B(R)\setminus B(R/2)}|\nabla u|^2 \dd x\right)^{\frac{1}{2}}\left(\avint_{B(R)}|\phi - [\phi]_{,R}|^4 \dd x\right)^{\frac{1}{4}}\times \\
    &\phantom{=}\times \left(\avint_{B(R)\setminus B(R/2)}|u - [u]_{B(R)\setminus B(B/2)}|^4 \dd x \right)^{\frac{1}{4}}\\
    &\phantom{=} + c(\varphi)\left|[u]_{B(R)\setminus B(B/2)}\right|\left(\int_{B(R)\setminus B(R/2)}|\nabla u|^2 \dd x\right)^{\frac{1}{2}}\left(\avint_{B(R)}|\phi - [\phi]_{,R}|^2 \dd x\right)^{\frac{1}{2}}\\
    & \leq c(\varphi)\|\phi\|_{BMO(\R^2)}\int_{B(R)\setminus B(R/2)}|\nabla u|^2 \dd x\\
    &\phantom{=} + c(\varphi)\|\phi\|_{BMO(\R^2)}\left|[u]_{B(R)\setminus B(B/2)}\right|\left(\int_{B(R)\setminus B(R/2)}|\nabla u|^2 \dd x\right)^{\frac{1}{2}};
\end{align*}
but as before, we deal with the term $\left|[u]_{B(R)\setminus B(B/2)}\right|$ in the following manner:
\begin{align}\label{E4.8}
\begin{split}
    \left|[u]_{B(R)\setminus B(B/2)}\right| &\leq \left|[u]_{B(R)\setminus B(R/2)} - [u]_{\varphi,R}\right| + \left|[u]_{\varphi,R}\right|\\
    &\leq c(\varphi)\left(\int_{B(R)}|\nabla u|^2 \dd x \right)^{\frac{1}{2}} + \frac{c(\varphi)}{R}\|\phi\|_{BMO(\R^2)}.
\end{split}
\end{align}
Consequently, we have
\[
I_2(R) \to 0\quad\mbox{ as }R\to \infty.
\]

Next, we estimate
\begin{align*}
    I_3(R) &\leq c(\varphi)\left(\int_{B(R)\setminus B(R/2)}|\pi_1|^2 \dd x \right)^{\frac{1}{2}}\left( \avint_{B(R)\setminus B(R/2)}|u - [u]_{B(R)\setminus B(R/2)}|^2 \dd x\right)^{\frac{1}{2}}\\
    &\phantom{=} + c(\varphi)|[u]_{B(R)\setminus B(R/2)}|\left(\int_{B(R)\setminus B(R/2)}|\pi_1|^2 \dd x \right)^{\frac{1}{2}};
\end{align*}
therefore, because $\pi_1 \in L_2(\R^2)$ and thanks to \eqref{E4.8}, we get
\[
I_3(R) \to 0\quad\mbox{ as }R\to \infty.
\]
Finally, since $\nabla \pi_2, \nabla (|b|^2/2) \in L_s(\R^2)$ with $1/s = 1/q + 1/2$, we have that $I_4(R)$ and $I_5(R)$ can be treated in the same way:
\begin{align*}
    I_4(R) &\leq \frac{c(\varphi,q)}{R^{\frac{2}{q}}}\left(\int_{B(R)\setminus B(R/2)} |\nabla \pi_2|^s\dd x \right)^{\frac{1}{s}}\left(\avint_{B(R)}|\phi - [\phi]_{,R}|^{s'}\dd x \right)^{\frac{1}{s'}}\\
    &\phantom{=} \to 0\quad \mbox{as }R\to \infty.
\end{align*}
Summarising our efforts, we obtain that
\[
\int_{\R^2}\left(|\nabla u|^2 + |\curl b|^2\right)\dd x \leq \varepsilon c(\varphi)\int_{\R^2}|\nabla u|^2 \dd x,
\]
for all $\varepsilon>0$; and the proof is done.
\end{proof}

\section{Concluding remarks}
At this point in time, it is still unclear to the authors whether one has a Liouville-type theorem for system \eqref{eq:mhd} assuming only condition \eqref{setup} to hold or, as suggested by Proposition \ref{prop:c0}, assuming in addition to condition \eqref{setup} that either $u\in BMO^{-1}(\R^2)$ or $b\in BMO^{-1}(\R^2)$. This will be the object of future investigations.
\appendix

\section{Technical lemmas}
\label{sec:appendix}

For the reader's convenience, we present three technical results needed in the proof our main results.

First, we recall an inequality from  \cite[Appendix B, Eq. (B.9), (B.10) \& (B.12)]{P-Lions96}. 

\begin{lemma}\label{lm:lions}
Let $f\in L_{2,loc}(\R^2)$ with finite Dirichlet energy
\[
\int_{\R^2}|\nabla f|^2 \dd x <\infty.
\]
Then
\[
\int_{B(R)}|f|^m \dd x \leq c R^2(\log 2R)^{\frac{m}{2}}\left( \int_{\R^2}|\nabla f|^2\dd x + \int_{B(1)}|f|^2\dd x\right)^{\frac{m}{2}},
\]
for all $m,R\geq 1$ and $c>0$ and absolute constant.
\end{lemma}

Secondly, we state an oscillation lemma inspired from \cite[Theorem 4.2]{Ser12_2}; although our statement of this result is slightly different, the idea of the proof is the same. With the aim of keeping the paper self-contained, we also provide a proof.
\begin{lemma}[Oscillation lemma] \label{lm:oscillation}
Let $v\in C(\overline{B})\cap H^1(B)$ such that for any $r\in ]0,1[$ the following maximum principle holds in $B(r)$:
\[
\max_{B(r)} v = \max_{\partial B(r)} v, \quad \min_{B(r)} v = \min_{\partial B(r)} v.
\]
Then,
\[
\sup_{x\in B(r)}|v(x) - v(0)| \leq \frac{c}{\sqrt{-\log r}}\|\nabla v\|_{L_2(B\setminus B(r))},
\]
for all $r\in ]0,1[$ and with $c>0$ an absolute constant.
\end{lemma}

\begin{proof}
For $r \in ]0,1[$, using radial and angular coordinates, we write
\begin{align}\label{eq:nabla2}
\begin{split}
    \int_{B \setminus B(r)} |\nabla v|^2 \dd x &= \int_r^{1}\int_0^{2\pi}|\nabla v(s,\theta)|^2 s \dd \theta \dd s \\
    &= \int_{r}^1 \int_0^{2\pi} \left[|\partial_s v|^2 + \frac{1}{s^2}|\partial_\theta v|^2\right] s \dd \theta \dd s \\
    &\ge \int_r^1 \frac{1}{s} \int_0^{2\pi} |\partial_\theta v(s,\theta)|^2 \dd \theta \dd s.
    \end{split}
\end{align}

On the other hand, by the fundamental theorem of calculus,
\begin{align*}
    v(s,\theta_1) - v(s,\theta_2) = \int_{\theta_1}^{\theta_2} \partial_\theta v(s,\theta) \dd \theta ,
\end{align*}
which implies (using the Cauchy-Schwarz inequality)
\begin{align*}
    |v(s,\theta_1) - v(s,\theta_2)| &\le \int_{0}^{2\pi} |\partial_\theta v(s,\theta)| \dd \theta \\
    &\le \sqrt{2\pi} \left(\int_0^{2\pi} |\partial_\theta v(s,\theta)|^2 \dd \theta\right)^{1/2}.
\end{align*}
Hence
\begin{align*}
    \max_{\theta_1 \in [0,2\pi[} v(s,\theta_1) -    \min_{\theta_2 \in [0,2\pi[} v(s,\theta_2) &\le \sqrt{2\pi}\left(\int_0^{2\pi} |\partial_\theta v(s,\theta)|^2 \dd \theta\right)^{1/2},
\end{align*}
which yields
\begin{align*}
    \mathrm{osc}_{\partial B(s)} v
   &\le \sqrt{2\pi}\left(\int_0^{2\pi} |\partial_\theta v(s,\theta)|^2 \dd \theta\right)^{1/2}.
\end{align*}
Using \eqref{eq:nabla2}, we deduce
\begin{align*}
    \int_r^1 (\mathrm{osc}_{\partial B(s)} v)^2 \frac{\dd s}{2\pi s} \le \int_{B \setminus B(r)} |\nabla v|^2 \dd x.
\end{align*}
We observe that $\mathrm{osc}_{\partial B(s)} v \ge \mathrm{osc}_{\partial B(r)}v \ge 0$ for $s \in [r,1]$. As a result,
\begin{align*}
    \int_{B \setminus B(r)} |\nabla v|^2 \dd x  &\ge \frac{1}{2\pi}\int_r^1 \frac{1}{s}  (\mathrm{osc}_{\partial B(s)} v)^2 \dd s \\
    &\ge \frac{(\mathrm{osc}_{\partial B(r)} v)^2}{2\pi}\int_r^1 \frac{1}{s}   \dd s
    \\&= \frac{1}{2\pi} (-\log r)(\mathrm{osc}_{\partial B(r)} v)^2.
\end{align*}
Once again, by the maximum principle, we have that $\mathrm{osc}_{B(r)} v = \mathrm{osc}_{\partial B(r)} v$, so
\begin{align*}
\mathrm{osc}_{B(r)} v
\leq \frac{\sqrt{2\pi}}{\sqrt{-\log r}}\|\nabla v\|_{L_2(B\setminus B(r))},
\end{align*}
which concludes the proof.

\end{proof}

Finally, we provide a technical lemma which was needed in the proof of Theorem \ref{thm:main}.

\begin{lemma}\label{lemma:construction}
Suppose $\psi_1,\psi_2$ are smooth functions such that $\nabla \psi_1 \in L_1(\mathbb{R}^2)$ and $\nabla \psi_2 \in L_\infty(\mathbb{R}^2)$. Then there exists a distributional solution $v \in L_\infty(\mathbb{R}^2)$ with $\nabla v\in  L_2(\mathbb{R}^2)$ to the equation
\begin{equation}\label{eq:laplace v eqn appendix}
    -\Delta v = \nabla^\perp \psi_1 \cdot \nabla \psi_2 \qquad \text{in } \mathbb{R}^2.
\end{equation}
Moreover, there exists a positive constant $c$ such that
\begin{equation}\label{eq:v bounds appendix}
    \Vert v \Vert_{L_\infty(\mathbb{R}^2)} + \Vert \nabla v \Vert_{L_2(\mathbb{R}^2)} \leq c \Vert \nabla \psi_1 \Vert_{L_1(\mathbb{R}^2)} \Vert \nabla \psi_2 \Vert_{L_\infty(\mathbb{R}^2)}.
\end{equation}
\end{lemma}
\begin{proof}
Recall our usual cut-off function $\varphi \in C^{\infty}_0(\R^2)$ such that $0\leq \varphi \leq 1$ with $\varphi \equiv 1$ in $B(1/2)$ and $\varphi \equiv 0$ in $B\setminus B(3/4)$. Now define, for each $n\in\mathbb{N}$, the rescaled cut-off function $\varphi_n(x) := \varphi(x/n)$, and
\begin{equation*}
    \hat{\psi}^n_1 := (\psi_1 - [\psi_1]_{,n}) \varphi_n , \qquad \hat{\psi}^n_2 := (\psi_2 - [\psi_2]_{,n}) \varphi_n.
\end{equation*}
We have that
\begin{equation*}
    v^n := (-\Delta)^{-1}(\nabla^\perp \hat{\psi}^n_1 \cdot \nabla \hat{\psi}^n_2)
\end{equation*}
is well-defined and
\begin{equation}\label{eq:equation laplace vn}
    -\Delta v^n = \nabla^\perp \hat{\psi}^n_1 \cdot \nabla \hat{\psi}^n_2 \qquad \text{in } \mathbb{R}^2.
\end{equation}

Our aim is now to obtain estimates on $v^n$ that are independent of $n$. We begin by observing that
\begin{equation}\label{eq:unif psi 1 bound in L1}
    \begin{aligned}
        \Vert \nabla \hat{\psi}^n_1 \Vert_{L_1(\mathbb{R}^2)} &\leq \Vert \nabla \psi_1 \Vert_{L_1(\mathbb{R}^2)} + \frac{c(\varphi)}{n} \int_{B(n)} |\psi_1 - [\psi_1]_{,n}| \, \d x \\
        &\leq c(\varphi) \Vert \nabla \psi_1 \Vert_{L_1(\mathbb{R}^2)},
    \end{aligned}
\end{equation}
where we used the triangle inequality and the boundedness of $\varphi$ and its derivative to obtain the first inequality, and the Poincar\'{e}--Wirtinger inequality on a ball to obtain the second inequality. Similarly, we have that
\begin{equation*}
    \Vert \nabla \hat{\psi}^n_2 \Vert_{L_\infty(\mathbb{R}^2)} \leq \Vert \nabla \psi_2 \Vert_{L_\infty(\mathbb{R}^2)} + \frac{c(\varphi)}{n}\sup_{x \in B(n)} \bigg| \psi_2(x) - \avint_{B(n)} \psi_2(y) \, \d y \bigg|,
\end{equation*}
where we used once again the triangle inequality and the boundedness of $\varphi$ and its derivative. Then, due to the bound on the gradient of $\psi_2$, we have that
\begin{align*}
    \frac{1}{n}\sup_{x \in B(n)} \bigg| \psi_2(x) - \frac{1}{|B(n)|}\int_{B(n)} \psi_2(y) \, \d y \bigg| &\leq \frac{1}{n}\sup_{x\in B(n)}\avint_{B(n)}|\psi_2 (x) - \psi_2(y)|\dd y \\
    &\leq c\|\nabla \psi_2\|_{L_{\infty}(\R^2)}, \qquad \forall n \in \mathbb{N},
\end{align*}
whence,
\begin{equation}\label{eq:unif psi 2 bound in Linfty}
    \Vert \nabla \hat{\psi}^n_2 \Vert_{L_\infty(\mathbb{R}^2)} \leq c(\varphi) \Vert \nabla \psi_2 \Vert_{L_\infty(\mathbb{R}^2)}, \qquad \forall n \in \mathbb{N}.
\end{equation}
Hence, the sequence $\{\nabla \hat{\psi}^n_1\}_{n\in\mathbb{N}}$ is uniformly bounded in $L_1(\mathbb{R}^2)$, and $\{\nabla \hat{\psi}^n_2\}_{n\in\mathbb{N}}$ is uniformly bounded in $L_\infty(\mathbb{R}^2)$.

In what follows, we directly estimate $v^n$ using its singular integral representation; these calculations are inspired from similar estimates due to Wente in \cite{wente}. By rewriting the right-hand side of \eqref{eq:equation laplace vn} in polar coordinates $(s,\theta)$, we find that
\begin{equation*}
    \nabla^\perp \hat{\psi}^n_1 \cdot \nabla \hat{\psi}^n_2 = \frac{1}{s}\partial_s (\hat{\psi}^n_1 \partial_\theta \hat{\psi}^n_2) - \frac{1}{s}\partial_\theta(\hat{\psi}^n_1 \partial_s \hat{\psi}^n_2),
\end{equation*}
and hence, using the periodicity with respect to the angular variable,
we get that the expression for $v^n(0)$ reads as
\begin{equation*}
    v^n(0) = - \frac{1}{2\pi}  \int_0^{2\pi} \bigg(\int_{0}^\infty \log s~ \partial_s (\hat{\psi}^n_1 \partial_\theta \hat{\psi}^n_2) \, \d s \bigg) \, \d\theta,
\end{equation*}
where we also applied the Fubini-Tonelli theorem to change the order of integration. In turn, after an integration by parts, this becomes
\begin{equation*}
    v^n(0) = \frac{1}{2\pi}  \int_0^{2\pi} \bigg(\int_{0}^\infty  \hat{\psi}^n_1 \frac{1}{s} \partial_\theta \hat{\psi}^n_2 \, \d s \bigg) \, \d\theta;
\end{equation*}
note that the boundary terms have vanished since
\[
|\log{s}~ \hat{\psi}^n_1 \partial_\theta \hat{\psi}^n_2| \leq \|\hat{\psi}^n_1\|_{L_{\infty}(\R^2)} \|\nabla \hat{\psi}^n_2\|_{L_{\infty}(\R^2)} |s \log s|\to 0\qquad \mbox{as }s\to 0^+,
\]
and because $\hat{\psi}^n_1$ and $\hat{\psi}^n_2$ are compactly supported. Using once again the Fubini-Tonelli theorem and the periodicity with respect to the angular variable in the $\partial_\theta \hat{\psi}^n_2$ portion of the integrand, the previous expression may be rewritten as
\begin{equation*}
    v^n(0) = \frac{1}{2\pi} \int_{0}^\infty \frac{1}{s} \bigg[\int_0^{2\pi}  \big( \hat{\psi}^n_1 - \frac{1}{2\pi}\int_0^{2\pi} \hat{\psi}^n_1(s,\alpha) \, \d \alpha \big)
    \partial_\theta \hat{\psi}^n_2 \, \d\theta\bigg] \, \d s.
\end{equation*}
In turn,
\begin{equation*}
    \begin{aligned}
        |v^n(0)| &\leq \frac{1}{2\pi} \Vert \nabla \hat{\psi}^n_2 \Vert_{L_\infty(\mathbb{R}^2)} \int_{0}^\infty \int_0^{2\pi} \bigg| \hat{\psi}^n_1 - \frac{1}{2\pi}\int_0^{2\pi} \hat{\psi}^n_1(s,\alpha) \, \d \alpha \bigg| \, \d s \, \d\theta \\
        &\leq c \Vert \nabla \hat{\psi}^n_2 \Vert_{L_\infty(\mathbb{R}^2)} \int_{0}^\infty \int_0^{2\pi} | \partial_\theta \hat{\psi}^n_1| \, \d s \, \d\theta \\
        &\leq c \Vert \nabla \hat{\psi}^n_2 \Vert_{L_\infty(\mathbb{R}^2)} \Vert \nabla \hat{\psi}^n_1 \Vert_{L_1(\mathbb{R}^2)},
    \end{aligned}
\end{equation*}
where we used the Poincar\'{e}-Wirtinger inequality to obtain the second inequality. By translation and using the uniform estimates \eqref{eq:unif psi 1 bound in L1} and \eqref{eq:unif psi 2 bound in Linfty}, we obtain
\begin{equation}\label{eq:unif vn bound Linfty}
    \Vert v^n \Vert_{L_\infty(\mathbb{R}^2)} \leq c \Vert \nabla \psi_2 \Vert_{L_\infty(\mathbb{R}^2)} \Vert \nabla \psi_1 \Vert_{L_1(\mathbb{R}^2)},
\end{equation}
for some positive constant $c$, independent of $n$. Additionally, by testing \eqref{eq:equation laplace vn} with $v^n \varphi_R$, where $\varphi_R$ is our usual rescaled cut-off function, we obtain that
\begin{equation*}
    \int_{\mathbb{R}^2} |\nabla v^n|^2 \varphi_R \, \d x = \frac{1}{2} \int_{\mathbb{R}} |v^n|^2 \Delta \varphi_R \, \d x + \int_{\mathbb{R}^2} (\nabla^\perp \hat{\psi}^n_1 \cdot \nabla \hat{\psi}^n_2) v^n \varphi_R \, \d x.
\end{equation*}
We thereby deduce from the uniform estimates \eqref{eq:unif psi 1 bound in L1}-\eqref{eq:unif vn bound Linfty} that there exists a positive constant $c$ independent of $n$ such that
\begin{equation}\label{eq:unif vn bound gradient L2}
    \Vert \nabla v^n \Vert_{L_2(\mathbb{R}^2)} \leq c \Vert \nabla \psi_2 \Vert_{L_\infty(\mathbb{R}^2)} \Vert \nabla \psi_1 \Vert_{L_1(\mathbb{R}^2)}.
\end{equation}
By the Banach-Alaoglu theorem, the uniform bounds \eqref{eq:unif vn bound Linfty} and \eqref{eq:unif vn bound gradient L2} imply the existence of a subsequence (which we still label as $\{v^n\}_{n\in\mathbb{N}}$) and a function $v$ such that
\begin{equation}\label{eq:weak convergences}
    v^n \overset{*}{\rightharpoonup} v \quad \text{in } L_\infty(\mathbb{R}^2), \qquad \nabla v^n \rightharpoonup \nabla v \quad \text{in } L_2(\mathbb{R}^2).
\end{equation}
Let us return now to the equation \eqref{eq:equation laplace vn}. We introduce the following test function $\zeta \in C^\infty_0(\mathbb{R}^2)$; there exists $R>0$ such that $\supp \zeta \subset B(R)$. We find that
\begin{equation*}
    \int_{\mathbb{R}^2} \nabla v^n \cdot \nabla \zeta \, \d x = \int_{\mathbb{R}^2} (\nabla^\perp \psi_1 \cdot \nabla \psi_2) \zeta \, \d x
\end{equation*}
for all $n \geq 2 R + 1$. We note that arguing in this way avoids having to directly consider the product of two weakly convergent sequences; on the right-hand side of the above.

Using \eqref{eq:weak convergences} to pass to the limit in the previous equation, we find that
\begin{equation*}
    \int_{\mathbb{R}^2} \nabla v \cdot \nabla \zeta \, \d x = \int_{\mathbb{R}^2} (\nabla^\perp \psi_1 \cdot \nabla \psi_2) \zeta \, \d x \qquad \forall \zeta \in C^\infty_0(\mathbb{R}^2),
\end{equation*}
and hence $v$ solves \eqref{eq:laplace v eqn appendix} in the sense of distributions. Estimate \eqref{eq:v bounds appendix} follows from the convergences \eqref{eq:weak convergences}, estimates \eqref{eq:unif vn bound Linfty}-\eqref{eq:unif vn bound gradient L2} and the weak (resp.~weak-*) lower semicontinuity of the norms.
\end{proof}


\vspace{5mm}
\section*{Acknowledgements}
N.~De Nitti was partially supported by the Alexander von Humboldt foundation and by the TRR-154 project of the Deutsche Forschungsgemeinschaft. F.~Hounkpe was funded by the Engineering and Physical Sciences Research Council [EP/L015811/1]. S.~Schulz was supported by the Royal Society [RGF/EA/181043]. We thank G.~M.~Coclite, G.~Seregin, and E.~Zuazua for helpful conversations on topics related to this work.

\vspace{5mm}


\vfill

\end{document}